\documentclass{amsart}
\usepackage{amscd}
\usepackage{amsmath}
\usepackage{amssymb}
\usepackage{amsthm}
\usepackage{graphicx}
\usepackage[all,cmtip]{xy}
\usepackage{fancyhdr}

\usepackage{qtree}

\theoremstyle{plain}
\newtheorem{thm}{Theorem}[section]
\newtheorem{lem}[thm]{Lemma}
\newtheorem{prop}[thm]{Proposition}

\newcommand{\Sth}{S^3}

\newcommand{\GK}{\Gamma_K}
\newcommand{\GDS}{\Gamma_{s}}
\newcommand{\GDF}{\Gamma_{f}}

\newcommand{\Hth}{\mathbb{H}^3}
\newcommand{\PSLTC}{PSL(2, \mathbb{C})}

\begin{document}

\title{On knot complements that decompose into regular ideal dodecahedra}
%\title{Not for circulation}

\author{Neil Hoffman}
\address{Max Planck Institute for Mathematics\\
Vivatsgasse 7\\
53111 Bonn\\
Germany}
\email{nhoffman@mpim-bonn.mpg.de
}

\date{\today}

\begin{abstract} 
Aitchison and Rubinstein constructed two knot complements that can be decomposed into two regular ideal dodecahedra. 
This paper shows that these knot complements are the only knot complements that decompose into $n$ regular ideal dodecahedra, providing a partial solution to a conjecture of Neumann and Reid.
\end{abstract}

\maketitle

\section{Introduction}\label{sect:intro}

% standard rep for commensurator -orientable and not
% standard rep for 333 and orientable and not
%  list of seven parabolic fixed points
% list of 6 conjugators
% unit argument and standard form argument 
% possible- include conjugate group argument
% possible - include explicit torsion argument
% possible - picture of horoballs. 

Coxeter observed that hyperbolic 3-space can be tessellated by regular ideal tetrahedra, cubes, octahedra, and dodecahedra (see \cite{Coxeter1956}).
Given that the figure 8 knot complement admits a 
geometric structure made up of two regular ideal tetrahedra (see \cite{Riley}), one may ask,``How many other knot complements can be decomposed into regular ideal polyhedra?" In \cite{AR}, Aitchison 
and Rubinstein examine a pair of knot complements that can be decomposed into regular ideal 
dodecahedra %(see Fig \ref{fig:dodecahedralknots})
 and described link complements that 
can be decomposed into regular ideal octahedra and cubes. The links described are well known. In fact, the complement of the Borromean rings are cited 
for the octahedral example and the complements of links $8_4 ^3$ and $8^4 _1$ are exhibited as the cubical examples. 
Using the computer software snap (see \cite{snapArt}), we can see 
that these links are in fact arithmetic, i.e. a conjugate of their respective fundamental groups in $\PSLTC$ shares a finite index subgroup with a Bianchi group. Since the figure 8 knot 
complement is only arithmetic knot complement (see \cite{Reid1}), to answer the question we only need to consider knot complements that decompose into regular 
ideal dodecahedra. 

This question is also of recent interest. Indeed, Neumann and Reid conjecture a stronger result by positing that the dodecahedral knots are the only knot complements admitting hidden symmetries (see \cite[Conj 1.3]{BBCW} and $\S$\ref{sect:prelim}). Furthermore, a negative answer to such a question would also provide evidence to a related conjecture of Reid and Walsh that every commensurability class contains at most three hyperbolic knot complements (see \cite[Conj 5.2]{RW} and $\S$\ref{sect:prelim}). Here, we say two finite volume hyperbolic 3-orbifolds are commensurable if they share a common finite sheeted cover. Two groups are commensurable if they are the fundamental groups of commensurable orbifolds and the equivalence relation on finite volume hyperbolic 3-orbifolds (or their corresponding fundamental groups) determined by commensurability partitions 3-orbifolds (or the associated groups) into commensurability classes.

%In order to tackle such a question about the dodecahedral knot complements, we first mention the following background. Neumann and Reid were able to show that the dodecahedral knot complements irregularly cover a rigid cusped orbifold and therefore, they admit hidden symmetries (see \cite[$\S$9-10]{NR1}). In fact in light of computer analysis of Goodman, Hodgson, and Heard, it appears quite rare for a knot complement to admit hidden symmetries (\cite{GHH}). The only other knot complement known to admit hidden symmetries is the figure 8 knot complement. The ability to cover a rigid cusped 

The main result of this paper provides a partial solution to the conjecture by classifying all knot complements that can be decomposed into regular ideal dodecahedra.

 \begin{thm}\label{mainthm}
 There are only two knot complements that can be decomposed into regular ideal dodecahedra. Furthermore, these are the only knot complements in their commensurability class.
 \end{thm}

We mention here that our notation for $\Gamma(5,2,2,6,2,3)$ and $\Gamma(5,2,2,3,3,3)$ of $\S$\ref{sect:GrpPresAndReps} differs slightly from that of \cite{NR1}, since they use $\Gamma^+$ to denote an orientation preserving subgroup. Also, the convention for denoting tessellations in \cite{NR1} differs as well. 
 
 \begin{proof}
Using \cite[$\S$9]{NR1}, the dodecahedral knot complements cover the orbifold \\
$\Hth/\Gamma(5,2,2,6,2,3)$. As noted in that section, $\Hth/\Gamma
(5,2,2,6,3,2)$ is the smallest volume orientable orbifold covered by the dodecahedral knot complements. Therefore, by \cite{Margulis1991},
 any orbifold in this commensurability 
class must cover $\Hth/\Gamma(5,2,2,6,2,3)$.  However, there are only two conjugacy classes of knot groups in $\Gamma(5,2,2,6,2,3)$ (see Lem 
\ref{lem:computerArgument}). This completes the proof.
\end{proof} 
 
 As evident from the argument above, the main theorem is a essentially a repackaging of Lemma \ref{lem:computerArgument}, which relies on a computer argument. Consequently, the rest of this paper will be devoted to proving this lemma. To this end, the paper is organized as follows. $\S$\ref{sect:prelim} gives a brief bit of background and establishes notation. $\S$\ref{sect:pReps} provides representations of both the symmetric dodecahedral knot group and commensurator in $\PSLTC$. In addition, presentations for the relevant knot groups are provided as well as a list of the parabolic elements used in the computer argument and two simplifying propositions.  $\S$\ref{sect:CompArgument} provides explanation of the computer argument and summarizes this argument via Lemma \ref{lem:computerArgument}. %The main theorem and two colloraries are given in $\S$\ref{mainResults}. 
%and the computer code is given the two appendices. 
Finally, we note that the computations were implemented in Magma V2.16-6. 

\textbf{Acknowledgments:} The author is grateful to Craig Hodgson, Alan Reid, and Genevieve Walsh for a number of helpful conversations and Sam Ballas for suggestions on an earlier version of the paper. Also, the author would like to thank the University of Texas for providing computer resources and Boston College for providing support for this project. 
%In addition, the author gratefully acknowledges the support of Boston College where much of this work was done.

% outline
%\begin{thm}
%The dodecahedral knot complements are the only knot complements in their commensurability class. 
%\end{thm}

\section{Preliminaries}\label{sect:prelim}
%\textbf{define?  - arithmetic, Bianchi group, commensurable,...}
\begin{figure}
\includegraphics[height=2.5in]{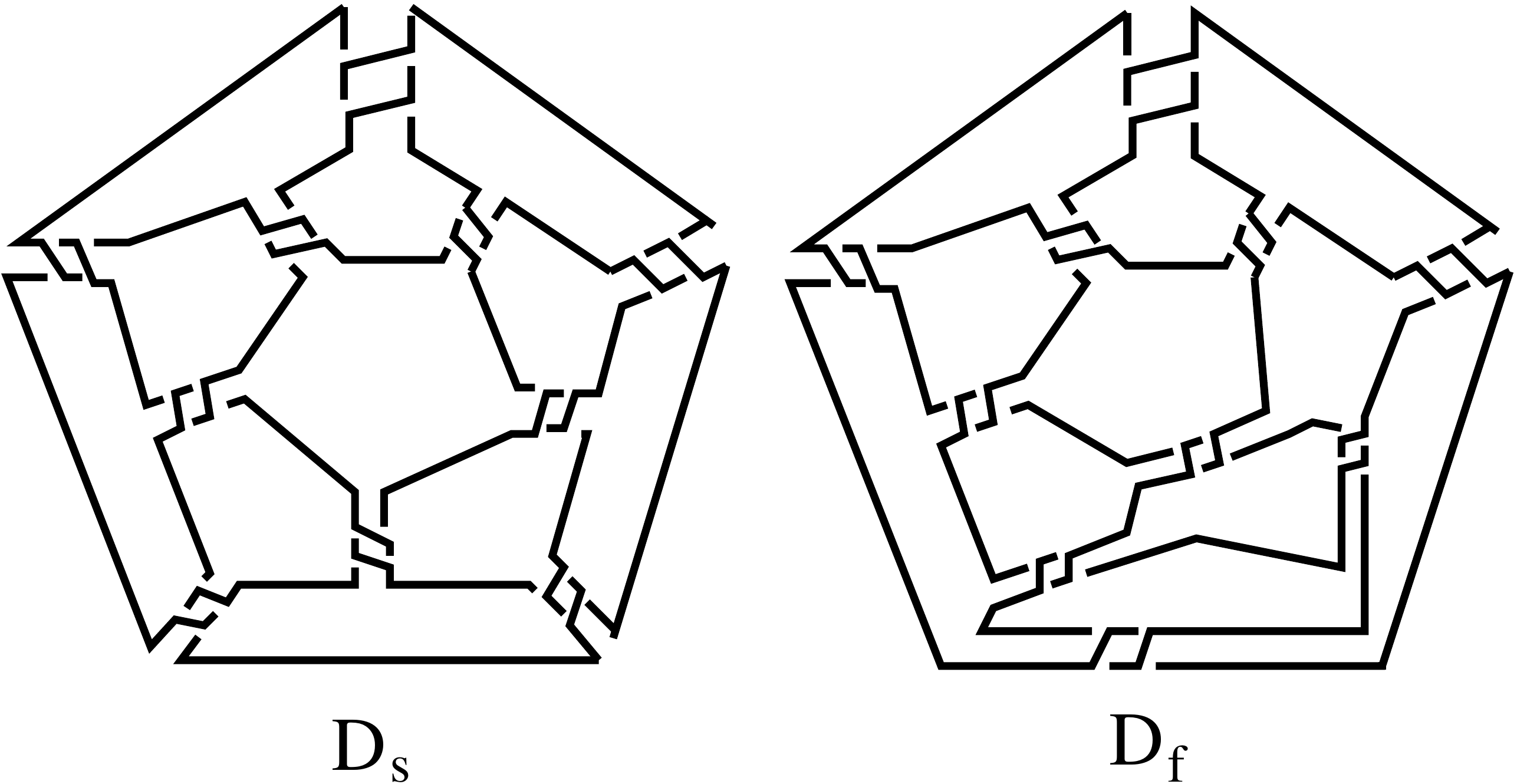}
\caption{\label{fig:dodecahedralknots} The dodecahedral knots of Aitchison and Rubinstein. These knot diagrams have been transcribed from their original figure (see \cite[Fig 9]{AR}).}
\end{figure}

We begin by providing background and establishing some conventions. %First, we discuss the dodecahedral knots.
Figure \ref{fig:dodecahedralknots} contains diagrams the two dodecahedral knots as presented in \cite{AR}. Following the notation of Aitchison and Rubinstein, we call these knots $D_s$ and $D_f$ as in the figure. The $D_s$ knot and corresponding knot complement have an orientation preserving symmetry group of $D_5$, while the complement of $D_f$ is fibered and has an orientation preserving symmetry group of order 2. These symmetries can be computed via snappy (see \cite{snappy}). We denote by $\GDS = \pi_1(\Sth-D_s)$ and $\GDF = \pi_1(\Sth-D_f)$. For convenience, we will refer to $\GDS$ and $\GDF$ collectively as the dodecahedral knot groups. Presentations and representations into $\PSLTC$ of these knot groups are provided later in $\S$\ref{sect:GrpPresAndReps}.

We will use the upper half space model of  $\Hth$ and identify $\partial \Hth$ with $\mathbb{C}$. Therefore, we use the natural identification $Isom^+(\Hth)=\PSLTC$.
 We say that the commensurator of a group $\Gamma$ in $\PSLTC$ denoted by $$Comm(\Gamma)=\{ g \in \PSLTC | g\Gamma g^{-1} \cap \Gamma \mbox{ is finite index in $\Gamma$ and $g\Gamma g^{-1}$} \}.$$ We say a group $\Gamma$ admits hidden symmetries if $|Comm(\Gamma):N(\Gamma)|> 1$. 
 
  In the arguments below, we will also consider the non-orientation preserving commensurator, however for our examples, this group will be generated by the commensurator and the element of $Isom(\Hth)$ corresponding to $\zeta \mapsto \bar{\zeta}$ ($\zeta \in \partial \Hth$). 
%Here, $Isom(\Hth)$ is acting on the standard model of upper half space and $w\in \partial \Hth=\mathbb{C}$.  

%As noted in $\S$\ref{sect:intro}, Margulis showed that the is $\Hth/Comm(\Gamma)$ is a finite volume 3-orbifold if and only if $\Gamma$ is not arithmetic (see \cite{Margulis1991}).
%In the case where $\Hth/Comm(\Gamma)$ is a finite volume 3-orbifold, the property of having hidden symmetries is equivalent to $\Hth/Comm(\Gamma)$ having a rigid cusp, ie the cusp is of the form 
%$O \times [0,\infty)$ where  $O$ is one of $S^2(2,4,4)$, $S^2(3,3,3)$, and $S^2(2,3,6)$. This is consistent with the standard definition of hidden symmetries as referenced in the previous section.
%
\subsection{Presentations and representations.} \label{sect:GrpPresAndReps}
As noted in \cite{NR1}, the commensurator of either dodecahedral knot complement group is the (orientation preserving) isometry group 
of the tessellation of $\Hth$ by regular ideal dodecahedra. Following the notation of \cite[page 144]{MR} together with the 
aforementioned paper, the commensurator of the dodecahedral knot groups is the tetrahedral group $\Gamma(5,2,2,6,2,3)$ 
which has $\Gamma(5,2,2,3,3,3)$ as an index 2 subgroup.
Again using \cite[page 144]{MR}, we obtain presentations for both of these groups.

$$
\Gamma(5,2,2,6,2,3) = \langle x',y',z' | x'^5,y'^2,z'^2,(y'z'^{-1})^6,(z'x'^{-1})^2,(x'y'^{-1})^3 \rangle$$
$$\Gamma(5,2,2,3,3,3) = \langle x,y,z | x^5,y^2,z^2,(zx^{-1})^3,(xy^{-1})^3,(yz^{-1})^3 \rangle
$$

To compute discrete faithful representations of these groups into $\PSLTC$, we follow the computation outlined in \cite{Lakeland}. Once again, the 
presentation for tetrahedral groups used in that paper differs slightly from the group presentation used above. To avoid confusion, we adapt these methods to be consistent with the presentations above. Also, the representations computed by Lakeland are normalized such that the horoballs tangent to $\infty$ (full sized horoballs in the language of Adams) have Euclidean diameter 1. Instead, our representations are normalized such that the parabolic elements that fix $\infty$ are of the from $\begin{pmatrix} 1 & n+m\omega \\ 0 & 1 \end{pmatrix}$ where $\omega = \frac{1+\sqrt{-3}}{2}$. One may convert between the representations described by Lakeland and the representations below by conjugating by $\beta = \begin{pmatrix} \frac{1}{\sqrt{u}} & 0 \\ 0 & \sqrt{u}  \end{pmatrix}$ where $u=2\cos \frac{\pi}{5}=\frac{1+\sqrt{5}}{2}$. %(aka the golden ratio). 

Finally, we can see that 
$x'=  \begin{pmatrix} u & -u^{-1} \\ u & 0 \end{pmatrix}$,
$y'=\begin{pmatrix} 0 & \omega\cdot u^{-1} \\ -\omega^{-1}\cdot u & 0 \end{pmatrix}$,
and $z'=\begin{pmatrix} 0 & iu^{-1}  \\ i u & 0 \end{pmatrix}$.

 More importantly for the arguments that follow, we give a representation of $\Gamma(5,2,2,3,3,3)$ corresponding to the 
 above presentation:
 $x= \begin{pmatrix} u & -u^{-1} \\ u & 0 \end{pmatrix}$,
 $y= \begin{pmatrix} 0 & \omega\cdot u^{-1} \\ -\omega^{-1}\cdot u & 0 \end{pmatrix}$, 
 and $z= \begin{pmatrix} 0 & \omega^2\cdot u^{-1}  \\ -\omega^{-2}\cdot u & 0 \end{pmatrix}$.

For the remainder of this argument, we will identify $\Gamma(5,2,2,3,3,3)$ with its image in $\PSLTC$.

%We also note that $\Gamma(5,2,2,6,2,3)$ is the orientation preserving commensurator group. We denote by $T_6$ and $T_3$ groups generated by reflections in the faces of hte 

\section{Parabolic elements in $\Gamma(5,2,2,3,3,3)$}\label{sect:pReps}

%\textbf{move this paragraph, add statements about propositions, perhaps move figure 2 into this section.}

%This section investigates a special set of groups generated by parabolic elements. Amazingly, the groups discussed 
%will be either knot groups with the same covolume of the dodecahedral knot complements or the groups will have torsion. 
This section introduces parabolic elements analyzed in Lemma \ref{lem:computerArgument} and provides two propositions that will help reduce the number of cases needed to consider in that argument. The parabolic elements are constructed by selecting a finite number of points in $\partial \Hth$ fixed by parabolic 
elements in the representation of $\Gamma(5,2,2,3,3,3)$ (see Figure \ref{fig:parabolicFixedPoints}). 
The choice is selected because it works for the computer assisted argument that follows. Since we are considering parabolic fixed points, the choice of these points is in no way unique. (For example, a different set could be generated by conjugating our representation of $\Gamma(5,2,2,3,3,3)$ by any element of $\PSLTC$ and looking at those fixed points.)
Although different sets of parabolic elements were tried, we make no claim that this set is the most efficient in terms of fewest cases needed to compute. 
%However, the two propositions in this section limit the number of computations we will need in the main argument. 

%First, we give two representations for the dodecahedral knots as subgroups of $\Gamma(5,2,2,3,3,3)$ listed above.

We now turn our attention to laying the foundation for the computer assisted computation 
that we will produce in the next section. As with any computer assisted computation, 
there are two (often competing) goals. First is to provide a computation that is thorough, accurate, and comprehensible, followed by an inherent desire to minimize 
size of the computation. However, this all relies on being able to reduce our problem to checking a finite number of cases. The following proposition establishes the 
fact that at a particular parabolic fixed point, there are only finitely many possible meridians to consider.

\begin{prop}\label{prop:meridianForm}
If $\GK$ is a knot group commensurable with $\Gamma(5,2,2,3,3,3)$, then the meridians of $\GK$ are of the form $g_p m_\infty g_p^{-1}$
where $g_p \in \Gamma(5,2,2,3,3,3)$ such that $g_p: p \mapsto \infty$ and $m_\infty =\begin{pmatrix} 1 & \pm \omega^i \\ 0 & 1 \end{pmatrix}$ such that $i$ an 
integer.
\end{prop}

\begin{proof}
Given the representation in $\S$\ref{sect:GrpPresAndReps}, the peripheral subgroup $P_\infty$ fixing $\infty$ is of the form: 
$\left\{ \left. \begin{pmatrix} 1 & n+m\omega \\ 0 & 1 \end{pmatrix}  \right| n,m \in \mathbb{Z} \right\}$. 
Also, there is a peripheral subgroup $P_0=y P_\infty y^{-1}$ fixing zero. Next, we again note that since $\GK$ is generated by parabolics, 
$\GK\subset \Gamma(5,2,2,6,2,3)$ and $\Gamma(5,2,2,3,3,3)$. Therefore any knot group, $\GK \subset \Gamma(5,2,2,3,3,3)$ has meridians 
fixing $0$ and $\infty$ ($m_0$ and $m_\infty$ 
respectively) and an element $g=\begin{pmatrix} a & b \\ c & d \end{pmatrix}$ that conjugates meridians fixing $\infty$ to meridians fixing $0$. Consideration of the 
(1,1) entry of $g$ shows $a=0$.  Following an identical argument to 
\cite[Lem 3.6]{RW}, we see that the meridians in $P_\infty$ and $P_0$ must have units in their off diagonal entries since $\GK$ is normally generated 
by a meridian and $g$ is non-trivial under any reduction homomorphism. This completes the proof.
\end{proof}

Next, we use snappy as a base for computations to obtain an unsimplified presentation for $\GDS$, \\

$\GDS=\langle a,b,c,d,e,f,g,h,i,j,k,l,m,n,o,p,q,r,s,t,u,v,w,x|b*v^{-1}*g^{-1}*o^{-1}*w, p*b^{-1}*c^{-1}, v*p*v^{-1}*u^{-1}*w*c^{-1}, r^{-1}*q*b^{-1}, k*t^{-1}*r, \
t*f*s^{-1}, a^{-1}*m*t*q^{-1}*k^{-1}, m*s*r^{-1}*k^{-1}, o*m^{-1}*a*w^{-1}, e*a^{-1}*t, e*m^{-1}*f^{-1}, e*o^{-1}*a^-\
1, l*u*c, j^{-1}*l^{-1}*s^{-1}*t, f*l^{-1}*w^{-1}*f^{-1}*u^{-1}, f*j*p*j^{-1}, q^{-1}*g*n^{-1}*x^{-1}, i^{-1}*h*x^{-1}*\
d^{-1}, h*n*v^{-1}*n^{-1}, i*g*x, u^{-1}*o*i^{-1}, d*o*d^{-1}*q, d*h^{-1}*o \rangle$.\\

We use the unsimplified presentation in order to control the simplification ourselves and find a presentation that is generated by parabolics. With some scratch work and use of Magma, we see this presentation could have a more efficient generating set. To that point, $a,f,h,o,$ and $u$ form a parabolic generating set for $\GDS$. These elements where used to make the choices for the first five parabolics later in this section. 

For the sake of completeness, we provide a presentation for the knot group $\GDF$: \\

$\GDF=\langle a,b,c,d,e,f,g,h,i,j,k,l,m,n,o,p,q,r,s,t,u,v,w|a^{-1}*v*a*h^{-1}, n*a^{-1}*w*h^{-1}, o*h^{-1}*d, o^{-1}*n*d^{-1}, w*c*h^{-1}*q^{-1}, i*c^{-1}*q*h*w^{-1}, i*v*c^{-1}, l^{-1}*w^{-1}*v*i, b*n^{-1}*c, g*l^{-1}*d^{-1}, s*g*e^{-1}*i^{-1}*a, s*b*l*e^{-1}*l^{-1}, p*r^{-1}*i^{-1}*v^{-1}, p^{-1}*r*e, f*r*q^{-1}, j^{-1}*f*g^{-1}, t^{-1}*j*s^{-1}*b^{-1}*k, q^{-1}*m^{-1}*j*k, t*m^{-1}*p^{-1}*q*e*f^{-1}, t*q*t^{-1}*p^{-1}, u*k^{-1}*h, m*u^{-1}*c^{-1}*o*u\rangle$.\\

In order to think of $\GDS$ as a subgroup of $\Gamma(5,2,2,3,3,3)$ in $\PSLTC$, we explicitly give the generators $a,f,h,o,u$ as words in $\Gamma(5,2,2,3,3,3)$.

$a = x*y^{-1}*z*y^{-1}$

$f= z*x*y^{-1}*z*y^{-1}*z$

$h= x*y^{-1}*x*y*x^{-1}*y^{-1}*x*y^{-1}*z*x*y^{-1}*x^{-1}*y*x^{-1}$

$o= x*y^{-1}*x^2*y^{-1}*z*y^{-1}*x^{-1}*y*x^{-1}$

$u= y*x*y^{-1}*z*x*y^{-1}*z*y^{-1}*z^{-1}*y^{-1}*x^{-1}*y^{-1}$ 

This computation will be used later to distinguish conjugacy classes of knot groups in $\Gamma(5,2,2,3,3,3)$ and $\Gamma(5,2,2,6,2,3)$.

We now describe a method for generating $3^{11}$  parabolic elements in $\Gamma(5,2,2,3,3,3)$ and $\Gamma(5,2,2,6,2,3)$. First, we note that for any discrete faithful representation of $\Gamma(5,2,2,6,2,3)$ 
into $\PSLTC$, any parabolic element of $\Gamma(5,2,2,6,2,3)$ is also a parabolic element of $\Gamma(5,2,2,3,3,3)$. Therefore any 
subgroup of $\Gamma(5,2,2,6,2,3)$ generated by parabolic elements will also be a subgroup of $\Gamma(5,2,2,3,3,3)$. In the 
upcoming search for subgroups of the commensurator that are knot groups, we only have to show that such groups are finite index in $\Gamma(5,2,2,3,3,3)$. 
%We use this assumption for ease and stability of computation. 

It is often convenient to consider a knot group as the normal closure of a meridian. However, for computations when the knot group is unknown 
this method becomes quite difficult to exploit. We get around this by observing in order to generate a knot group, we only need a finite number of elements conjugate to the meridian in the knot group. This conjugation would also exist in the commensurator $\Gamma(5,2,2,6,2,3)$ and $\Gamma(5,2,2,3,3,3)$. Hence, there are only six possible conjugates  of $\begin{pmatrix} 1 & 1 \\  0 & 1 \end{pmatrix}$ at a given fixed point and since we are interested in groups generated by parabolic elements we only have to consider these parabolic up to to inverses. A bit of care should be exercised here to avoid confusion. Since $\Hth/\Gamma(5,2,2,3,3,3)$ has one cusp, given two parabolics $p_1$ and $p_2$ conjugate to $\begin{pmatrix} 1 & 1 \\  0 & 1 \end{pmatrix}$ in $\Gamma(5,2,2,6,2,3)$ and fixing points $r_1$ and $r_2$, respectively, there is an element $g$ in $\Gamma(5,2,2,3,3,3)$ with $g(r_2)=r_1$ such that either $p_1=g p_2 g^{-1}$ or $p_1=g p_2^{-1} g^{-1}$.  Therefore, after fixing a finite of parabolic fixed points, we can list all conjugates of a meridian fixing $\infty$ in $\Gamma(5,2,2,3,3,3)$ and consideration of only the groups generated by finite sets of these conjugates in $\Gamma(5,2,2,3,3,3)$
will result in the same set of groups as consideration of conjugates  in $\Gamma(5,2,2,6,2,3)$.

First, we provide the three possible meridians fixing $\infty$.
We will also provide the identification with the representation of $\Gamma(5,2,2,3,3,3)$ as listed in $\S$\ref{sect:GrpPresAndReps}. 

% a in magma code
$m_{\infty,0}=x*y^{-1}*z*y^{-1}=\begin{pmatrix} 1 & 1 \\  0 & 1 \end{pmatrix}$

% t in magma code
$m_{\infty,1}=z*x^{-1}*z*y^{-1}=\begin{pmatrix} 1 & \omega \\  0 & 1 \end{pmatrix}$

% t*a^{-1} in magma code
$m_{\infty,2}=z*x^{-1}*y*x^{-1}=\begin{pmatrix} 1 & \omega^{2} \\  0 & 1 \end{pmatrix}$

Using these meridians, we can refine the explanation above.  As noted in Proposition \ref{prop:meridianForm}, all meridians of knot groups (up to inverses) are of the form $g_p m_{\infty,0} g_p^{-1}$, $g_p m_{\infty,1} g_p^{-1}$, and $g_p m_{\infty,2} g_p^{-1}$, where $g_p:p\mapsto \infty$. This is really just a way to conjugate $\begin{pmatrix} 1 & 1 \\  0 & 1 \end{pmatrix}$ by $g_p$, $g_p\cdot r$, $g_p \cdot r^2$ where $r$ is a rotation of order 3 fixing $\infty$.  However, by first rotating an then conjugating, it makes creating an exhaustive list of possible meridians at parabolic fixed point much more feasible. 
Thus, we provide a relevant set of $g_p$'s below, which we will use in the proof of the Lemma \ref{lem:computerArgument}.

Next, consider $g_{p_1} = y$, then $p_1=0$ and there are three meridians fixing 0 we should consider are: $m_{p_1,0}=y*m_{\infty,0}*y^{-1}$, $m_{p_1,1}=y*m_{\infty,1}*y^{-1}$,
and $m_{p_1,2}=y*m_{\infty,2}*y^{-1}$. (Note that $m_{\infty,0}*y*m_{\infty,2}*y^{-1}$ is an element of order 5). In general, we will use $m_{p,j} = g_p m_{\infty,i} g_p^{-1}$ for all $p\ne \infty$.

Figure \ref{fig:parabolicFixedPoints} defines the remaining nine $g_p$'s used in the proof. The first column is the index i, second column provides $g_{p_i}$ and the third column provides $p_i$. 

\begin{center}
\begin{figure}
\begin{tabular}{|c|c|c|}
\hline
\hline
$i$ & $g_{p_i}$ & $p_i$\\
\hline
1 & $y$ & $0$\\
2 & $x*y*x^{2}*y*x$ & $-\frac{\omega}{u^2}$\\
3 & $m_{\infty,2}^{-1}*y$ & $\omega^{-1}$\\
4 & $y*m_{\infty,0}*y$ & $\frac{1}{\omega(1+u)}$\\
5 & $m_{\infty,0}*m_{0,0}$ & $1+\frac{1}{\omega(1+u)}$\\
6 & $m_{p_3,1}$ & $\frac{\omega}{u}-1+\frac{2}{\omega}$\\
7 & $m_{p_1,2}$ & $\frac{-1}{u^2}$\\
8 & $m_{\infty,1}*m_{1,0}$ & $\frac{w}{u}+\frac{1}{u^2}$\\
9 & $m_{\infty,1}*y$ & $\omega$\\
10 & $x$ & $1$\\
\hline
\end{tabular}
\caption{\label{fig:parabolicFixedPoints} Together with $\infty$ and $0$, the parabolic fixed points in this table are the only ones used in the computer argument that follows.}
\end{figure}
\end{center}

In light of Proposition \ref{prop:meridianForm}, there are 9 possible pairs of meridians (up to inverses) comprised of a parabolic fixing 0 and parabolic fixing $\infty$.
 We provide the following argument in order to whittle down the possible sets we need to consider.

\begin{prop}\label{prop:ZeroAndInfinity}
In $\Gamma(5,2,2,3,3,3)$, any knot group is conjugate to group with elements:
 $\begin{pmatrix} 1 & 1 \\ 0 & 1 \end{pmatrix}$ and $\begin{pmatrix} 1 & 0\\ -u^2\cdot \omega^2 & 1 \end{pmatrix}$.
\end{prop}

\begin{proof}
As seen in the previous proof, $m_\infty =\begin{pmatrix} 1 & \pm \omega^i \\ 0 & 1 \end{pmatrix}$ and $m_0=\begin{pmatrix} 1 & 0 \\ \mp\omega^j u^2 & 1 \end{pmatrix}$ with $-2\leq i,j \leq 3$. 

After conjugating $\GK$ by a sufficient power of $yz^{-1}$, we may assume $\begin{pmatrix} 1 & 1 \\ 0 & 1 \end{pmatrix}\in \GK$ and $m_0 = \begin{pmatrix} 1 & 0 \\ \mp \omega^j u^2 & 1 \end{pmatrix}$ with $j=\pm 1,0$. 

However, $tr \begin{pmatrix} 1 & 1 \\  0 & 1 \end{pmatrix}\begin{pmatrix} 1 & 0 \\  -u^2 & 1 \end{pmatrix} = 2-u^2=\frac{-1}{u}=2\cdot \cos(\frac{3\pi}{5})$
 and therefore is an elliptic element of order 5. Thus, $j=\pm 1$ and using the element in the full commensurator that is realized by complex conjugation on the entries of elements, we may assume $j=1$.

Finally, This element acting by complex conjugation is normalizes $\Gamma(5,2,2,3,3,3)$ in $Isom(\Hth)$,  so it does not affect the conjugacy classes of the knot group.
\end{proof}

\section{Computer verification}\label{sect:CompArgument}
This section outlines the computation done in Magma for the main argument. Ultimately, it exhausts all possible sets of meridians fixing the chosen set of fixed points 
(see $\S$\ref{sect:pReps}) and shows that the groups they generate are all of finite index in the commensurator of the dodecahedral knot complements. Moreover,
the indices that result from this computation allow us to classify all of the subgroups of $\Gamma(5,2,2,3,3,3)$ and $\Gamma(5,2,2,6,2,3)$ up to conjugacy.

We will see the main argument can verified by repeatedly using only a few of Magma's standard functions. 
(An executable file in magma that performs the necessary computations is available on the author's website.)
First, $\Gamma(5,2,2,3,3,3)$ is entered as a finitely presented group or \emph{FPGroup} in the parlance of Magma. Next, we generate groups using sets of possible 
meridians as given by Figure \ref{fig:parabolicFixedPoints}. For each group generated by candidate meridians, we check the index of the group in $\Gamma(5,2,2,3,3,3)$ using the \emph{Index(FPGroup, FPGroup)} function of Magma. In the cases where the index comes out to be greater than 1, we use the \emph{quo$\langle$ FPGroup $|$ relation $\rangle$} to obtain the quotient of our group by the normal closure of a meridian and \emph{Order(FPGroup)} to verify that such a 
quotient group is trivial. By Perelman's positive solution to the Poincare conjecture (see \cite{MT}), this condition is equivalent to being a knot group. When a 
knot group is discovered, we use the \emph{IsCongugate(FPGroup, FPGroup, FPGroup)} to check to see if the second and third arguments are conjugate in the group 
given by the first argument. (An explicit example of this is provided in the paragraph above the statement of Lemma \ref{lem:computerArgument}.) This process establishes the conjugacy classes in $\Gamma(5,2,2,3,3,3)$. The final step is to input $\Gamma(5,2,2,6,2,3)$ obtain $\Gamma(5,2,2,3,3,3)$ as an index 2 subgroup and consider conjugacy classes of knot groups in $\Gamma(5,2,2,6,2,3)$. 

Proposition \ref{prop:ZeroAndInfinity} shows we only need to consider generating sets that include $m_{\infty,0}$ and $m_{p_1,1}$. A priori, there would be $3^9$ other generating sets to consider. However, in reality many fewer cases are needed. Figures \ref{fig:TableFirstPart} and \ref{fig:TableSecondPart} give tables of the cases need for a complete argument.

The data is encoded as follows: under columns labeled by $\infty$, 1-10, we record the index of the meridian used in the generating set. The column marked index records the index of the group generated in $\Gamma(5,2,2,3,3,3)$. In  most cases, it is either 1 or 60 with 60 indicating (after computing the normal closure of a meridian) that the group is a knot group. In one case, we use the symbol $T$ to indicated that the group generated is not shown to be finite index, but a simple computation of $tr(m_{p_1,1} \cdot m_{p_4,1})$ shows the group has torsion of order 5. In that case, the group generated can not be a knot group nor the subgroup of a knot group. 

As mentioned above, the computer argument chooses a finite set of fixed points and shows that any collection of them generates a finite index subgroup of $\Gamma(5,2,2,3,3,3)$. For example, the first line of Figure \ref{fig:TableSecondPart} records that $\langle m_{\infty,0}, m_{p_1,1}, m_{p_2,1}, m_{p_3,2} \rangle$ generates a group of index 1 in $\Gamma(5,2,2,3,3,3)$. Hence, in this case, we do not have to consider adding elements $m_{p_4,j}$.
Furthermore, by consideration of the last two lines of Figure \ref{fig:TableFirstPart} and this first line of Figure \ref{fig:TableSecondPart}, we know that any subgroup of $\Gamma(5,2,2,3,3,3)$ with meridian candidates at each parabolic fixed point including the set of elements $\langle m_{\infty,0}, m_{p_1,1}, m_{p_2,1} \rangle$ can not be a knot group because any element of form $m_{p_3,j}$ will force such a group to be index 1 in $\Gamma(5,2,2,3,3,3)$. In this manner, we check that any knot group, which must have a meridian at each parabolic fixed point, is index 60 in $\Gamma(5,2,2,3,3,3)$ and 36 of the 45 possible generating sets of meridians we consider actually generate all of $\Gamma(5,2,2,3,3,3)$.

Finally, we explain the last column of the table. There are three conjugacy classes of knot groups in $\Gamma(5,2,2,3,3,3)$ and two in $\Gamma(5,2,2,6,2,3)$. The three conjugacy classes in the first group are marked by $I$, $IIA$, and $IIB$ with 
$IIA$ and $IIB$ being the same conjugacy class in 
$\Gamma(5,2,2,6,2,3)$. The conjugacy class $I$ contains groups isomorphic to $\GDS$ and $IIA$ or $IIB$  contains groups isomorphic to $\GDF$. For example, if we denote by 
$$H_1=\langle m_{\infty,0}, m_{1,1}, m_{2,0}, m_{3,0}, m_{4,2}, m_{5,0}, m_{6,0}, m_{7,2} \rangle$$ 
and $$H_2=\langle m_{\infty,0}, m_{1,1}, m_{2,0}, m_{3,0}, m_{4,2}, m_{5,0}, m_{6,1} \rangle,$$ then we see that 
\emph{IsCongute($\Gamma(5,2,2,3,3,3)$,$H_1$,$H_2$)} returns false, while\\ \emph{IsCongute($\Gamma(5,2,2,6,2,3)$,$H_1$,$H_2$)} returns true. The latter result indicates that these groups are in fact isomorphic and the corresponding knot complements are homeomorphic.  

We summarize the above explanation in the following lemma.

\begin{lem}\label{lem:computerArgument}
\begin{enumerate}
\item Up to conjugation in $\Gamma(5,2,2,3,3,3)$, there are three subgroups that are knot groups. Furthermore, these groups are index 60 in $\Gamma(5,2,2,3,3,3)$.

\item Up to conjugation in $\Gamma(5,2,2,6,2,3)$, there are two subgroups that are knot groups. Furthermore, these groups are index 120 in $\Gamma(5,2,2,6,2,3)$.
\end{enumerate}
\end{lem}  

Finally, results of Magma computations come in two flavors. 
Either Magma has enough information to answer a question or the computation requires too much memory and 
Magma returns an indefinite answer. Given these possibilities, we remark that Magma is more reliable dealing with computations in groups of lower index. Without 
recognizing this, it would seem natural to consider all knot groups as subgroups of $\Gamma(5,2,2,6,2,3)$. This approach seems especially appealing in light of the 
fact that there are two conjugacy classes of $\GDF$ in  $\Gamma(5,2,2,6,2,3)$. 
However, the \emph{Index} and \emph{Order} functions of Magma return inconclusive answers for many of the cases when working inside
 of $\Gamma(5,2,2,6,2,3)$. Thus, we must work inside of $\Gamma(5,2,2,3,3,3)$.

\begin{center}
\begin{figure}
\begin{tabular}{||c|c||c|c|c|c|c|c|c|c|c||c|c||}
\hline
\hline
$\infty$ & 1 & 2 & 3 & 4 & 5 & 6 & 7 & 8 & 9 & 10 & Index & Conj. class\\ 
\hline
%H01000 is infinite index but adding m5i is index 1
0 & 1 &  0 & 0 & 0 & 0 & - & - & - &- & - & 1 & -\\
0 & 1 &  0 & 0 & 0 & 1 & - & - & - &- & - & 1 & -\\
0 & 1 &  0 & 0 & 0 & 2 & - & - & - &- & - & 1 & -\\
%H01001 is index 1
0 & 1 &  0 & 0 & 1 & - & - & - & - &- & - & 1 & -\\
%H01002 is not finite index
% the next bit is the telescoping
0 & 1 &  0 & 0 & 2 & 0 & 0 & 0 & - &- & - & 60 & I\\
0 & 1 &  0 & 0 & 2 & 0 & 0 & 1 & - &- & - & 1 & -\\
0 & 1 &  0 & 0 & 2 & 0 & 0 & 2 & - &- & - & 60 & IIA\\
0 & 1 &  0 & 0 & 2 & 0 & 1 & - & - &- & - & 60 &  IIB\\
0 & 1 &  0 & 0 & 2 & 0 & 2 & - & - &- & - & 1 & -\\
% H01002E1 is index 1
0 & 1 &  0 & 0 & 2 & 1 & - & - & - &- & - & 1 & -\\
% H01002E2 is not finite index
0 & 1 &  0 & 0 & 2 & 2 & 0 & - & - &- & - & 60 & IIB\\
0 & 1 &  0 & 0 & 2 & 2 & 1 & - & - &- & - & 60 & IIA\\
0 & 1 &  0 & 0 & 2 & 2 & 2 & - & - &- & - & 1 & -\\
0 & 1 &  0 & 1 & 0 & - & - & - & - &- & - & 60 & I\\   
0 & 1 &  0 & 1 & 1 & - & - & - & - &- & - & 1 & -\\   
% H01012 requires telescoping
0 & 1 &  0 & 1 & 2 & 0 & - & 0 & - &- & - & 60 & IIA\\   
0 & 1 &  0 & 1 & 2 & 0 & - & 1 & - &- & - & 1 & -\\   
0 & 1 &  0 & 1 & 2 & 0 & - & 2 & - &- & - & 60 & IIB\\   
0 & 1 &  0 & 1 & 2 & 1 & - & - & - &- & - & 1 & -\\   
0 & 1 &  0 & 1 & 2 & 2 & - & - & - &- & - & 1 & -\\   
0 & 1 &  0 & 2 & 0 & - & - & - & - &- & - & 1 & -\\   
0 & 1 &  0 & 2 & 1 & - & - & - & - &- & - & 1 & -\\   
0 & 1 &  0 & 2 & 2 & - & - & - & - &- & - & 1 & -\\   
0 & 1 &  1 & 0 & - & - & - & - & - &- & - & 1 & -\\   
%0 & 1 &  1 & 0 & 0 & - & - & - & - &- & - & 1 & -\\   
%0 & 1 &  1 & 0 & 1 & - & - & - & - &- & - & 1 & -\\   
%0 & 1 &  1 & 0 & 2 & - & - & - & - &- & - & 1 & -\\   
%
0 & 1 &  1 & 1 & - & - & - & - & - &- & - & 1 & -\\   
%0 & 1 &  1 & 1 & 0 & - & - & - & - &- & - & 1 & -\\   
%0 & 1 &  1 & 1 & 1 & - & - & - & - &- & - & 1 & -\\   
%0 & 1 &  1 & 1 & 2 & - & - & - & - &- & - & 1 & -\\   
%
\hline
\end{tabular}
\caption{\label{fig:TableFirstPart} The indices of the groups generated by elements conjugate to $m_{\infty,0}^{\pm 1}$ in $\Gamma(5,2,2,3,3,3)$.}
\end{figure}
\end{center}

\begin{center}
\begin{figure}
\begin{tabular}{||c|c||c|c|c|c|c|c|c|c|c||c|c||}
\hline
\hline
$\infty$ & 1 & 2 & 3 & 4 & 5 & 6 & 7 & 8 & 9 & 10 & Index & Conj. class\\ 
\hline
0 & 1 &  1 & 2 & - & - & - & - & - &- & - & 1 & -\\   
%
%0 & 1 &  1 & 2 & 0 & - & - & - & - &- & - & 1 & -\\   
%0 & 1 &  1 & 2 & 1 & - & - & - & - &- & - & 1 & -\\   
%0 & 1 &  1 & 2 & 2 & - & - & - & - &- & - & 1 & -\\   
%
0 & 1 &  2 & 0 & - & - & - & - & - &- & - & 1 & -\\   
%
%0 & 1 &  2 & 0 & 0 & - & - & - & - &- & - & 1 & -\\   
%0 & 1 &  2 & 0 & 1 & - & - & - & - &- & - & 1 & -\\   
%0 & 1 &  2 & 0 & 2 & - & - & - & - &- & - & 1 & -\\   
%H01210E0
0 & 1 &  2 & 1 & 0 & 0 & 0 & 0 & - &- & - & 1 & -\\   
0 & 1 &  2 & 1 & 0 & 0 & 0 & 1 & - &- & - & 1 & -\\   
0 & 1 &  2 & 1 & 0 & 0 & 0 & 2 & 0 & 0 & - & 1 & -\\   
0 & 1 &  2 & 1 & 0 & 0 & 0 & 2 & 0 & 1 & 0 & 1 & -\\   
0 & 1 &  2 & 1 & 0 & 0 & 0 & 2 & 0 & 1 & 1 & 1 & -\\   
0 & 1 &  2 & 1 & 0 & 0 & 0 & 2 & 0 & 1 & 2 & 1 & -\\   
0 & 1 &  2 & 1 & 0 & 0 & 0 & 2 & 0 & 2 & - & 1 & -\\   

0 & 1 &  2 & 1 & 0 & 0 & 0 & 2 & 1 & - & - & 1 & -\\   
0 & 1 &  2 & 1 & 0 & 0 & 0 & 2 & 2 & - & - & 1 & -\\   

%H1210E01
0 & 1 &  2 & 1 & 0 &0& 1 &  - & - &- & - & 1 & -\\   
%H1210E02
0 & 1 &  2 & 1 & 0 &0& 2 &  - & - &- & - & 1 & -\\   

0 & 1 &  2 & 1 & 0 & 1 & - & - & - &- & - & 1 & -\\   
0 & 1 &  2 & 1 & 0 & 2 & - & - & - &- & - & 1 & -\\   
%
% H1211
0 & 1 &  2 & 1 & 1 & - & - & - & - &- & - & 1 & -\\   
% H1212
0 & 1 &  2 & 1 & 2 & - & - & - & - &- & - & 1 & -\\   
% H1220
0 & 1 &  2 & 2 & 0 & - & - & - & - &- & - & 1 & -\\   
% H1221
0 & 1 &  2 & 2 & 1 & - & - & - & - &- & - & T & -\\   
%H1222
0 & 1 &  2 & 2 & 2 & - & - & - & - &- & - & 1 & -\\

\hline
\end{tabular}
\caption{\label{fig:TableSecondPart} The indices of the groups generated by elements conjugate to $m_{\infty,0}^{\pm 1}$ in $\Gamma(5,2,2,3,3,3)$ continued.}
\end{figure}
\end{center}

\bibliographystyle{plain}
\bibliography{NRHBib}

\end{document}